\documentclass[11pt]{article}

\usepackage{amsthm}
\usepackage{amsfonts}
\usepackage{amsmath}
\usepackage{amssymb}
\usepackage{color}
\usepackage{easybmat}
\usepackage{blkarray}
\usepackage{datetime}
\usepackage{hyperref}
\usepackage{comment}
\usepackage{accents}

\usepackage[T2A]{fontenc}

\usepackage{mathrsfs}
\renewcommand{\mathcal}[1]{\mathscr{#1}}

\usepackage{marvosym}

\usepackage{theory}
\usepackage{theorems}

\newcommand{\ta}{\oplus}
\newcommand{\tp}{\odot}
\newcommand{\td}{\oslash}

\renewcommand{\phi}{\varphi}

\newcommand{\bb}[1]{\mathbb{#1}}

\newcommand\restr[2]{{
  \left.\kern-\nulldelimiterspace 
  #1 
  \vphantom{\big|} 
  \right|_{#2} 
}}

\ddmmyyyydate
\title{Tropical Combinatorial Nullstellensatz\\ and Sparse Polynomials\footnote{An extended abstract of a preliminary version~\cite{GrigorievP17} appeared in the proceedings of the 21st International Symposium on Fundamentals of Computation Theory (FCT 2017).\newline
The results of Sections 4 and 6 were obtained by the first author  at MCCME and supported by the Russian Science Foundation (project 16-11-10075). The results of Sections 3 and 5 were obtained by the second author and were supported by grant MK-5379.2018.1, by the
Russian Academic Excellence Project `5-100' and by RFBR grant 17-51-10005-KO\_a. }}

\author{Dima Grigoriev \and Vladimir V. Podolskii}

\author{Dima Grigoriev$^1$, Vladimir V. Podolskii$^{2,3}$\\[3pt]
$^1$\small CNRS, Math\'ematiques, Universit\'e de Lille, Villeneuve d'Ascq, 59655, France\\
\small \href{mailto:Dmitry.Grigoryev@math.univ-lille1.fr}{Dmitry.Grigoryev@math.univ-lille1.fr}\\
$^2$ \small Steklov Mathematical Institute, Moscow, Russia\\
$^3$ \small National Research University Higher School of Economics, Moscow, Russia\\
 \small   \href{mailto:podolskii@mi.ras.ru}{podolskii@mi.ras.ru}
}

\date{}

\begin{document}

\maketitle

\begin{abstract}
Tropical algebra emerges in many fields of mathematics such as algebraic geometry, mathematical physics and combinatorial optimization. In part, its importance is related to the fact that it makes various parameters of mathematical objects computationally accessible. Tropical polynomials play a fundamental role in this, especially for the case of algebraic geometry. On the other hand, many algebraic questions behind tropical polynomials remain open.
In this paper we address four basic questions on tropical polynomials closely related to their computational properties:
\begin{enumerate}
\item Given a polynomial with a certain support (set of monomials) and a (finite) set of inputs, when is it possible for the polynomial to vanish on all these inputs?
\item A more precise question, given a polynomial with a certain support and a (finite) set of inputs, how many roots can this polynomial have on this set of inputs?
\item Given an integer $k$, for which $s$ there is a set of $s$ inputs such that any non-zero polynomial with at most $k$ monomials has a non-root among these inputs?
\item How many integer roots can have a one variable polynomial given by a tropical algebraic circuit?
\end{enumerate}
In the classical algebra well-known results in the direction of these questions are Combinatorial Nullstellensatz due to N.~Alon, J.~Schwartz - R.~Zippel Lemma and Universal Testing Set for sparse polynomials respectively. The classical analog of the last question is known as $\tau$-conjecture due to M.~Shub - S.~Smale.
In this paper we provide results on these four questions for tropical polynomials.

\end{abstract}


\tableofcontents

\section{Introduction}

A \emph{max-plus} or a \emph{tropical semiring} is defined by a set $\bb{K}$, which can be $\mathbb{R}$ or $\mathbb{Q}$
endowed with two operations, the \emph{tropical addition} $\ta$ and the \emph{tropical multiplication} $\tp$, defined in
the following way:
$$
x \ta y = \max\left(x,y\right), \ \ \ \ x \tp y = x + y.
$$

Tropical polynomials are a natural analog of classical polynomials.
In classical terms a tropical polynomial is an expression of the form $f(\vec{x}) = \max_i M_i(\vec{x})$,
where each $M_i(\vec{x})$ is a linear polynomial (a tropical monomial) in variables $\vec{x} = (x_1, \ldots, x_n)$,
and all the coefficients of all $M_i$'s are nonnegative integers except for constant terms that can be any elements of $\bb{K}$ (the constant term corresponds to a coefficient of the tropical monomial and other coefficients correspond to the powers of variables in the tropical monomial). 

The degree of a tropical monomial $M$ is the sum of its coefficients (except the constant term) and
the degree of a tropical polynomial $f$ denoted by $\deg(f)$ is the maximal degree of its monomials.
A point $\vec{a} \in \bb{K}^n$ is a root of the polynomial $f$ if the maximum $\max_i\{M_i(\vec{a})\}$
is attained on at least two different monomials $M_i$.
The detailed definitions on the basics of max-plus algebra are provided in Preliminaries.

Tropical polynomials have appeared in various areas of mathematics and found many applications (see, for example,~\cite{IMS2009tropical,MS2015tropical,sturmfels02equations,Mikhalkin2004survey,RGST05first_steps,HuberS95polyhedral,Vorobyev67}). An early source of the tropical approach was the Newton's method for solving
algebraic equations in Newton-Puiseux series~\cite{sturmfels02equations}.
An important advantage of tropical algebra is that it makes some properties of classical mathematical objects computationally accessible~\cite{theobald06frontiers,IMS2009tropical,MS2015tropical,sturmfels02equations}:
on one hand tropical analogs reflect certain properties of classical objects and on the other hand tropical objects have much more simple and discrete structure and thus are more accessible to algorithms.
One of the main goals of max-plus mathematics is to build a theory of tropical polynomials which would help to work with
them and would possibly lead to new results in related areas. Computational applications, on the other hand, make
it important to keep the theory maximally computationally efficient.

The case studied best so far is the one of tropical linear polynomials and systems of tropical linear polynomials.
For them an analog of a large part of the classical theory of linear polynomials was established.
This includes studies of tropical analogs of the rank of a matrix and the independence of vectors~\cite{DSS05rank,IzhakianR2009rank,AkianGG09rank},
an analog of the determinant of a matrix and its properties~\cite{AkianGG09rank,DSS05rank,Grigoriev13complexity}, an analog of Gauss triangular form~\cite{Grigoriev13complexity}. Also the solvability problem for tropical linear systems was studied from the complexity point of view.
Interestingly, this problem turns out to be polynomially equivalent to the mean payoff games problem~\cite{akian12mean_payoff,GP13Complexity} which received considerable attention in computational complexity theory.

For tropical polynomials of arbitrary degree less is known. In~\cite{IzhakianS07} the radical of a tropical
ideal was explicitly described.
In~\cite{RGST05first_steps,ST10SIAM} a tropical version of the Bezout theorem was proved for tropical polynomial systems for the case when the number of polynomials in the system is equal to the number of variables.
In~\cite{DavydowG16} the Bezout bound was extended to systems with an arbitrary
number of polynomials.
In~\cite{GrigorievP18} the tropical analog of Hilbert's Nullstellensatz was established.
In~\cite{Bihan16} a bound on the number of nondegenerate roots of a system of sparse tropical polynomials was given.
In~\cite{theobald06frontiers} it was shown that the solvability problem for tropical polynomial systems is $\NP$-complete.

\paragraph{Our results.} In this paper we address several basic questions for tropical polynomials.

The first question we address is given a set $S$ of points in $\bb{R}^n$ and a set of monomials of $n$ variables, is there a tropical polynomial with these monomials that has roots in all the points of the set. In the classical case a famous result in this direction with numerous applications in Theoretical Computer Science and in Number Theory is the Combinatorial Nullstellensatz~\cite{Alon99}. Very roughly, it states that the set of monomials of a polynomial can be substantially larger than the set $S$ of the points and at the same time the polynomial is still non-zero on at least one of the points in $S$. In the tropical case we show that this is not the case: if the number of monomials is larger than the number of points, there is always a polynomial with roots in all the points. We establish the general criterion for existence of a polynomial on a given set of monomials with roots in all the points of a given set (Theorem~\ref{thm:comb_general} below). From this criterion we deduce that if the number of points is equal to the number of monomials, and the set of points and the set of monomials are structured in the same way (more specifically, these sets augmented with coordinate-wise order are isomorphic), then there is no polynomial with roots in all the points (Theorem~\ref{thm:combinatorial_null}). We note that the last statement for the classical case is an open question~\cite{Risler90}.

There is one more notable difference of our version of Combinatorial Nullstellensatz compared to the classical case. In the classical version an important technical assumption in the theorem is that a certain large degree monomial occurs in the polynomial. Without this assumption the classical theorem is not true: there might be a polynomial with zeros in all points of a certain set and with a small number of monomials. In the tropical case on the other hand once the polynomial has roots in some set of points, we can add any monomials to this polynomial without reducing the number of zeros.

The second question is given a finite set $T \subseteq \bb{R}$ how many roots can a tropical polynomial of $n$ variables and degree $d$  have in the set $T^n$? In the classical case the well-known Schwartz-Zippel lemma~\cite{Zippel79,Schwartz80} states that the maximal number of roots is $d|T|^{n-1}$. We show that in the tropical case the maximal possible number of roots is $|T|^n - (|T|-d)^n$ (Theorem~\ref{thm:schwartz-zippel}).
We note that this result can be viewed as a generalization and improvement of isolation lemma of Mulmuley, Vazirani, and Vazirani~\cite{Mulmuley87,CRS95isolation,Klivans01,Ta-Shma15}. In particular, we prove a more precise version of a technical result in~\cite[Lemma~4]{Klivans01}. The paper~\cite{Ta-Shma15} proves the same upper bound as in our result for the special case of $d=1$.

The third question is related to a \emph{universal testing set} for tropical polynomials of $n$ variables with at most $k$ monomials. A universal testing set is a set of points $S \subseteq \bb{K}^n$ such that any nontrivial polynomial with at most $k$ monomials has a non-root in one of the points of $S$. The problem is to find a minimal size of a universal testing set for given $n$ and $k$. In the classical case this problem is tightly related to the problem of interpolating a polynomial with a certain number of monomials (with a priori unknown support) given its values on some universal set of inputs. The classical problem was studied in~\cite{GrigorievK87,Ben-OrT88,KaltofenY89,GrigorievKS91} and the minimal size of the universal testing set for the classical case turns out to be equal to $k$, in particular, independent from $n$ (while for the interpolation problem the size is $2k$). In the tropical case it turns out that the answer depends on which tropical semiring $\bb{K}$ is considered: for $\bb{K}=\bb{R}$ we show that as in the classical case the minimal size of a universal testing set is equal to $k$ (Theorem~\ref{thm:univ_set_r}). For $\bb{K}=\bb{Q}$ it turns out that the minimal size of a universal testing set is substantially larger. We show that its size is $\Theta(kn)$ (Theorems~\ref{thm:upper bound on k}~and~\ref{thm:univ_set_nonconstructive}; the constants in $\Theta$ do not depend on $k$ and $n$)\footnote{For two non-negative real valued functions $f(k,n)$ and $g(n,k)$ the notation $f = \Theta(g)$ means that there are positive constants $c$ and $C$ such that $cf(k,n) \leq g(k,n) \leq Cf(k,n)$ for all $k$ and $n$.}. For $n=2$ we find the precise size of a minimal universal testing set $s = 2k-1$ (Theorems~\ref{thm:univ_dim_2_upper}~and~\ref{thm:univ_set_dim_2}). For greater $n$ the precise minimal size of a universal testing set remains unclear. Finally, we establish an interesting connection of this problem to the following problem in Discrete Geometry: what is the minimal number of disjoint convex polytopes in $n$-dimensional space that is enough to cover any set of $s$ points in such a way that all $s$ points are on the boundaries of the polytopes (Theorem~\ref{thm:covering_connection} and Corollary~\ref{cor:covering1}~and Lemma~\ref{lem:covering2}).

The fourth question is related to the number of integer roots of a single-variable polynomial computed by an algebraic circuit. In the classical case a well known $\tau$-conjecture states that the number of integer roots of a single-variable polynomial computed by an algebraic circuit is upper bounded by a polynomial in the size of the circuit~\cite{ShubS95,Smale98,Blum98} (see~\cite{KoiranPT15,KoiranPTT15} for some recent developments). The positive answer to this conjecture would imply an algebraic version of $\P \neq \NP$ statement. The conjecture is open even for the case of algebraic formulae. We address a tropical analog of this conjecture. Interestingly, in the tropical case the answer is different for tropical formulae and tropical circuits. We observe that if a tropical polynomial of one variable is computed by a tropical formula, then the number of roots of this polynomial is upper bounded by the size of the formula (Lemma~\ref{lem:formula}). On the other hand, we show that for circuits the tropical analog of $\tau$-conjecture is false: there is a family of tropical polynomials of one variable that are computable by tropical circuits of linear size and have exponentially many integer roots (Theorem~\ref{thm:tropical_circuits}). For the proof of this result we adapt a construction from~\cite{montufar2014}.




The rest of the paper is organized as follows. 
In Section~\ref{sec:prelim} we introduce necessary definitions and notations. 
In Section~\ref{sec:comb_null} we give the results on a tropical analog of Combinatorial Nullstellensatz.
In Section~\ref{sec:schwartz_zippel} we prove a tropical analog of Schwartz-Zippel Lemma. 
In Section~\ref{sec:univ_set} we give the bounds on tropical universal sets. 
In Section~\ref{sec:tau} we prove results on the tropical analog of $\tau$-conjecture.
\section{Preliminaries} \label{sec:prelim}

A \emph{max-plus} or a \emph{tropical semiring} is defined by a set $\bb{K}$ (which we take to be $\mathbb{R}$ or
$\mathbb{Q}$ in the present paper) 
endowed with two operations, the \emph{tropical addition} $\ta$ and the \emph{tropical multiplication} $\tp$, defined in
the following way:
$$
x \ta y = \max\{x,y\}, \ \ \ \ x \tp y = x + y.
$$

A tropical (or max-plus) monomial in variables $\vec{x} = (x_1, \ldots, x_n)$ is defined as
\begin{equation} \label{eq:monomial}
m(\vec{x}) = c \tp x_1^{\tp i_1} \tp \ldots \tp x_n^{\tp i_n},
\end{equation}
where $c$ is an element of the semiring $\bb{K}$ and $i_1, \ldots, i_n$ are nonnegative integers.
In the usual notation the monomial is the linear function
$$
m(\vec{x}) = c + i_1 x_1 + \ldots + i_n x_n.
$$
For $\vec{x} = (x_1, \ldots, x_n)$ and $I = (i_1, \ldots, i_n)$
we introduce the notation
$$
\vec{x}^I = x_1^{\tp i_1} \tp \ldots \tp x_n^{\tp i_n} = i_1 x_1 + \ldots + i_n x_n.
$$
The degree of the monomial $m$ is defined as the sum $i_1 + \ldots + i_n$. We denote this sum by $|I|$.

A \emph{tropical polynomial} is the tropical sum of tropical monomials
$$
p(\vec{x}) = \bigoplus_i m_i(\vec{x})
$$
or in the usual notation $p(\vec{x}) = \max_i m_i(\vec{x})$.
The \emph{degree} of the tropical polynomial $p$ denoted by $\deg(p)$ is the maximal degree of its monomials.
A point $\vec{a} \in \bb{K}^n$ is a \emph{root} of the polynomial $p$ if the maximum $\max_i\{m_i(\vec{a})\}$
is attained on at least two distinct monomials among $m_i$ (see e.g.~\cite{RGST05first_steps} for the motivation of this definition). 
A polynomial $p$ \emph{vanishes} on the set $S \subseteq \bb{K}^n$ if all the points of $S$ are roots of $p$. A polynomial $p$ is \emph{vanishing identically} if it has no monomials.

Geometrically, a tropical polynomial $p(\vec{x})$ is a convex piece-wise linear function and the roots of $p$ are non-smoothness points of this function.

By the \emph{product} of  two tropical polynomials $p(\vec{x}) = \bigoplus_i m_i(\vec{x})$ and $q(\vec{x}) = \bigoplus_j m_j'(\vec{x})$ we naturally call a tropical polynomial $p \tp q$ that has as monomials tropical products $m_i(\vec{x}) \tp m_j'(\vec{x})$ for all $i,j$. We will make use of the following simple observation.

\begin{lemma} \label{lem:tropical_polynomials_product}
A point $\vec{a} \in \bb{K}^n$ is a root of $p \tp q$ iff it is a root of $p(\vec{x})$ or $q(\vec{x})$.
\end{lemma}
\begin{proof}
Suppose $\vec{a}$ is a root of $p$. Let $m_{i_1}(\vec{x}), m_{i_2}(\vec{x})$ be two distinct monomials of $p$ such that $m_{i_1}(\vec{a})=m_{i_2}(\vec{a}) = \max_i m_{i}(\vec{a})$. Let $m_{j_1}'(\vec{x})$ be a monomial of $q$ such that $m_{j_1}'(\vec{a})= \max_j m_{j}'(\vec{a})$. Then $m_{i_1} \tp m_{j_1}'$ and $m_{i_2} \tp m_{j_1}'$ are two distinct monomials of $p\tp q$ with the maximal value on $\vec{a}$ among all the monomials of $p \tp q$. The symmetrical argument shows that any root of $q$ is a root of $p \tp q$.

Iff $\vec{a}$ is not a root neither of $p$ nor of $q$, then there are unique $i_1$ and $j_1$ such that $m_{i_1}(\vec{a}) = \max_i m_{i}(\vec{a})$ and $m_{j_1}'(\vec{a})= \max_j m_{j}'(\vec{a})$. Then the maximal value on $\vec{a}$ among all the monomials of $p \tp q$ is attained on a single monomial $m_{i_1} \tp m_{j_1}'$ and thus $\vec{a}$ is not a root of $p \tp q$.
\end{proof}

For two vectors $\vec{a},\vec{b} \in \bb{R}^n$ throughout the paper we will denote by $\langle \vec{a}, \vec{b} \rangle$ their inner product.

\section{Tropical Combinatorial Nullstellensatz} \label{sec:comb_null}

For a polynomial $p$ denote 
by $\Supp(p)$ the set of all  $J=(j_1,\ldots,j_n)$ such that the  monomial $\vec{x}^J$ occurs in $p$. 

Consider two finite sets $S, R \subseteq \mathbb{R}^n$ such that $|S|=|R|$.
We call $S$ and $R$ \emph{non-singular} if there is a bijection $f\colon S \to R$ such that $\sum_{x\in S} \langle \vec{x}, f(\vec{x}) \rangle$ is greater than the corresponding sum for all other bijections from $S$ to $R$. Otherwise we say that $R$ and $S$ are \emph{singular}. Note that the notion of singularity is symmetrical.

First we formulate a general criterion for vanishing polynomials with a given support.

\begin{theorem} \label{thm:comb_general}
Consider a (finite) support $S \subseteq \mathbb{N}^n$ and a (finite) set of points $R \subseteq \mathbb{K}^n$. There are three cases.
\renewcommand*{\theenumi}{\thetheorem(\roman{enumi})}%
  \renewcommand*{\labelenumi}{(\roman{enumi})}%

\begin{enumerate}
\item \label{thm:comb_general1} If $|R| < |S|$, then there is a polynomial $p$ with support in $S$ vanishing on $R$.
\item \label{thm:comb_general2} If $|R|=|S|$, then there is a polynomial $p$ with support in $S$ vanishing on $R$ iff $S$ and $R$ are singular.
\item \label{thm:comb_general3} If $|R|>|S|$ then there is a polynomial $p$ with support in $S$ vanishing on $R$ iff for any subset $R'\subset R$ such that $|R'|=|S|$ we have that $R'$ and $S$ are singular.
\end{enumerate}

\begin{remark}
Before we proceed to the proof of the theorem we observe that in Theorem~\ref{thm:comb_general1} we can have not only a polynomial $p$ with $\Supp(p)\subseteq S$, but also a polynomial with the property $\Supp(p)=S$. Indeed, if some monomials with exponent vector in $S$ are missing in $\Supp(p)$, we can add them with small enough coefficients, so that the value of this monomial is smaller than the maximal values of monomials in $p$ on all points of $R$ (recall, that $R$ is finite).
\end{remark}

\end{theorem}

\begin{proof}
Consider a polynomial 
$$
p(\vec{x}) = \bigoplus_{J \in S} c_J \tp \vec{x}^{J}
$$
with support $S$.
The claim that $p$ has a root in $\vec{a} \in R$ means that the maximum in 
$$
\max_{J\in S} \left(c_{J} + \langle J, \vec{a} \rangle\right)
$$
is attained on at least two monomials $J_1$ and $J_2$.
Note that once $S$ and $R$ are fixed, this claim is a linear tropical equation on the coefficients $\{c_J\}_{J \in S}$ of $p$.

The claim that $p$ has a root in all the points of $R$ thus means that the coefficients of $p$ satisfy a tropical linear system with the matrix 
$$
\left(\langle J, \vec{a}\rangle\right)_{J \in S, \vec{a} \in R} \in \bb{R}^{|S|\times |R|}.
$$

The tropical Cramer rule~\cite[Theorem 5.3]{RGST05first_steps} states that if the number of rows $|R|$ in such system is less than the number of columns $|S|$, then there is always a solution. Thus, in this case there is a polynomial with roots in all the points of $R$.

If the matrix is square, that is $|R|=|S|$, then it is known~\cite[Lemma 5.1]{RGST05first_steps} that there is a solution iff the tropical determinant of the matrix is singular. Tropical determinant is a tropicalization of the classical one. That is for our matrix it is given by the expression
$$
\bigoplus_{f\colon S \to R} \left( \bigodot_{J \in S}  \langle J, f(J)\rangle \right) = \max_{f\colon S \to R} \left( \sum_{J \in S}  \langle J, f(J)\rangle \right),
$$
where $f$ ranges over all bijections from $S$ to $R$.
Its singularity means that the maximum is attained on at least two different monomials. This means that there are two bijections $f,g \colon S \to R$ with equal maximum sum $\sum_{J \in S}  \langle J, f(J)\rangle = \sum_{J \in S}  \langle J, g(J)\rangle$. Note that this is precisely the singularity of $S$ and $R$.

If the number of rows $|R|$ in the matrix is greater than the number of columns $|S|$ in it, then it is known~\cite{AkianGG09rank,DSS05rank,Grigoriev13complexity,IzhakianR2009rank} that the system has a nontrivial solution iff the tropical determinant of each square submatrix of size $|S|\times |S|$ is singular. This means precisely that for any subset $R^\prime \subset R$ such that $|R^\prime| = |S|$ the sets $R^\prime$ and $S$ are singular.
\end{proof}

Now we will derive corollaries of this general criterion.

Suppose we have a set $S \subseteq \bb{N}^n$.
Suppose also we have a set of reals $\{\alpha^{i}_{j}\}$ for $i=1,\ldots, n$, $j \in \bb{N}$ such that for each $i$ we have 
$$
\alpha^i_0 < \alpha^i_1 < \alpha^i_2 < \ldots . 
$$

For $J=(j_1,\ldots,j_n)$ we introduce the notation $\vec{\alpha}_{J} = (\alpha^{1}_{j_1},\ldots,\alpha^{n}_{j_n})$.
Consider the set $R_S = \{\vec{\alpha}_{J} \mid J \in S\}$.

\begin{remark} 
The key example for this definition is the case $\alpha_{j}^i=j$ for all $j$ and $i$. In this case $R_{S}=S$. For $S=\mathbb{N}^n$ this set is just the set of vertices of integer lattice in $n$-dimensional space. In the general case the set $R_{\mathbb{N}^n}$ is just a distorted version of this grid, where the distortion is performed in each dimension independently.
\end{remark}

We consider the following question. Suppose we have a polynomial $p$ with the support $\Supp(p) \subseteq S$. For which sets $S' \subseteq \bb{N}^n$ is it possible that $p$ vanishes on $R_{S'}$?





A natural question is the case of $S=S'$.
We show the following theorem.

\begin{theorem} \label{thm:combinatorial_null}
For any $S$ and for any non-vanishing identically tropical polynomial $p$ such that $\Supp(p) \subseteq S$ there is $\vec{r} \in R_S$ such that $\vec{r}$ is a non-root of $p$.
\end{theorem}

An interesting case of this theorem is $S = \{0, 1,\ldots, k\}^n$. Then the result states that any non-zero polynomial of individual degree at most $k$ w.r.t. each variable $x_i$, $i=1,\ldots,n$, does not vanish on a lattice of size $k+1$.

Theorem~\ref{thm:comb_general1} and Theorem~\ref{thm:combinatorial_null} answer some customary cases of our first question. We note that the situation here is quite different from the classical case. The classical analog of Theorem~\ref{thm:combinatorial_null} for the case of $S =  \prod_{i=1}^n\{0, 1,\ldots, k_i\}$ is a simple observation. In the tropical setting it already requires some work. On the other hand, in the classical case it is known that for such $S$ the domain of the polynomial can be substantially larger then $S$ and still the polynomial remains non-vanishing on $R_S$ (see Combinatorial Nullstellensatz~\cite{Alon99}). In tropical case, however, if we extend the domain of the polynomial even by one extra monomial, then due to Theorem~\ref{thm:comb_general1} there is a vanishing non-zero polynomial.

In the proof of Theorem~\ref{thm:combinatorial_null} we will use the following simple technical lemma, that is essentially from~\cite[p.~261]{hardy1988inequalities}. We provide a proof for the sake of completeness.

\begin{lemma} \label{lem:perm_technical}
Consider two sequences of reals $v_1\leq v_2 \leq \ldots \leq v_l$ and $u_1\leq u_2 \leq \ldots \leq u_l$. Consider any permutation $\sigma \in Sym_l$ on $l$ element set. Then
$$
\sum_i v_i u_i \geq \sum_i v_i u_{\sigma(i)}.
$$
Moreover, the inequality is strict iff there are $i,j$ such that $v_i < v_j$, $u_{\sigma(j)} < u_{\sigma(i)}$.

\end{lemma}


\begin{proof}
We count the number of inversions in $\sigma$: $D=|\{(i,j)\mid i<j, \sigma(j)<\sigma(i)\}|$. We show the lemma by induction on $D$.
For the step of induction we pick one inversion $(i,j)$ and swap it.
We observe that by this we do not introduce new inversions.
 
We then use the following observation: if $a\leq b$ and $c \leq d$, then
$$
bd + ac \geq bc + da.
$$
This inequality holds since it is equivalent to $ (b-a)(d-c)\geq 0$.

We also observe that the inequality is strict iff both inequalities $a\leq b$ and $c \leq d$ are strict.

So, after the swap of $i$ and $j$ the sum in the statement of the lemma does not decrease.
Thus the inequality follows.

To prove the second part of the lemma, if there is a pair $i,j$ as stated in the lemma, just switch $i$ and $j$ on the first step. By this we get the strict inequality.
If there is no such a pair $i,j$, note that we do not introduce one during the process above since we do not introduce new inversions.
\end{proof}

\begin{proof}[Proof of Theorem~\ref{thm:combinatorial_null}]
By Theorem~\ref{thm:comb_general} it is enough to show that $S$ and $R_S$ are non-singular.

Consider the bijection $f \colon S \to R_S$ given by $f(J) = \vec{\alpha}_{J}$. We claim that the maximum over all possible bijections $g$ of the sum $\sum_{J\in S} \langle J, g(J) \rangle$ is attained on the bijection $f$ and only on it.

Consider an arbitrary bijection $g \colon S \to R_S$.  Since $R_S \subseteq \bb{R}^n$ it is convenient to denote $g(J) = (g_1(J),\ldots,g_n(J))$ and $f(J) = (f_1(J),\ldots,f_n(J))$. Consider the sum 
$$
\sum_{J\in S} \langle J, g(J) \rangle = \sum_{J\in S} \sum_{i=1}^n j_i g_i(J) = \sum_{i=1}^n \sum_{J\in S} j_i g_i(J)
$$
We will show that for each $i$ 
\begin{equation} \label{eq:comb_1_app}
\sum_{J\in S} j_i g_i(J) \leq \sum_{J\in S} j_i f_i(J)
\end{equation}
and for at least one $i$
\begin{equation} \label{eq:comb_2_app}
\sum_{J\in S} j_i g_i(J) < \sum_{J\in S} j_i f_i(J)
\end{equation}
From these inequalities the theorem follows.

Take an arbitrary $i$ and consider projections of all the points in the set $S$ on the $i$-th coordinate. Enumerate these projections in the nondecreasing order:
$$
j_{1,1}=\ldots=j_{1,k_1} < j_{2,1}=\ldots=j_{2,k_2}<\ldots<j_{l,1}=\ldots=j_{l,k_l}.
$$
Different points in $S$ can have the same $i$-th coordinate, so we split points into blocks according to their $i$-th coordinate.
Due to the definition of $R_S$, the projections of its points on the $i$-th coordinate will have the same structure:
$$
r_{1,1}=\ldots=r_{1,k_1} < r_{2,1}=\ldots=r_{2,k_2}<\ldots<r_{l,1}=\ldots=r_{l,k_l}.
$$

Both bijections $f$ and $g$ induce bijections $f'$ and $g'$ from the sequence $\vec{j}$ to the sequence $\vec{r}$. Moreover $f$ induces a natural bijection: $f^\prime(j_{i_1,i_2})=r_{i_1,i_2}$. The inequality~\eqref{eq:comb_1_app} thus follows from the first part of Lemma~\ref{lem:perm_technical}.

For inequality~\eqref{eq:comb_2_app} note that since $g\neq f$ there is $J \in S $ such that $g(J) \neq \vec{\alpha}_{J}$. This means that there is $i$ such that 
$$
g_i(J) \neq \alpha^i_{j_i}.
$$ 
Thus for the bijection induced by $g$ on the coordinate $i$ we have that $g^{\prime}(j_{i_1,i_2}) = r_{i_1^\prime,i_2^\prime}$, where $i_1 \neq i_{1}^\prime$. Without loss of generality assume that $i_1 < i_{1}^\prime$, the opposite case is symmetrical. Consider the subsequence 
$$
\vec{j}'=j_{i_1^\prime,1}, \ldots, j_{i_1^\prime,k_{i_1^\prime}}, \ldots, j_{l,1}, \ldots, j_{l,k_l}.
$$
Since $j_{i_1,i_2}$ is mapped by $g^\prime$ into the sequence $\vec{j}'$ and $g^\prime$ is a bijection, there is $j_{i_3,i_4}$ in $\vec{j}'$ that is mapped outside of this sequence, that is $g(j_{i_3,i_4}) = r_{i_{3}^\prime,i_4^\prime}$, where $i_{3}^\prime < i_1^\prime$.

Denoting $a = j_{i_1,i_2}$ and $b = j_{i_3,i_4}$ we obtain that $a < b$, but $g^\prime(a) > g^\prime(b)$. By Lemma~\ref{lem:perm_technical} this gives inequality~\eqref{eq:comb_2_app}.
\end{proof}

\section{Tropical Analog of Schwartz-Zippel Lemma} \label{sec:schwartz_zippel}

Using the results of the previous section we can prove an analog of Schwartz-Zippel Lemma for tropical polynomials.

\begin{theorem} \label{thm:schwartz-zippel}
Let $S_1, S_2, \ldots, S_n \subseteq \bb{K}$, denote $|S_i|=k_i$.
Then for any $d \leq \min_i k_i$ the maximal number of roots a non-vanishing identically tropical polynomial $p$ of degree $d$ can have in $S_1\times \ldots \times S_n$ is equal to
$$
\prod_{i=1}^{n} k_i - \prod_{i=1}^n \left(k_i-d\right).
$$
Exactly the same statement is true for polynomials with the individual degree in each variable at most $d$.
\end{theorem}

In particular, we have the following corollary.
\begin{corollary}
Let $S \subseteq \bb{K}$ be a set of size $k$.
Then for any $d \leq k$ the maximal number of roots a non-vanishing identically tropical polynomial $p$ of degree $d$ can have in $S^n$ is equal to
$$
k^n - (k-d)^n.
$$
Exactly the same statement is true for polynomials with the individual degree in each variable at most $d$.
\end{corollary}

\begin{proof}[Proof of Theorem~\ref{thm:schwartz-zippel}]

The upper bound is achieved on the product of $d$ linear polynomials. Indeed, denote $S_i = \{s_{i,1}, s_{i,2}, \ldots, s_{i,{k_i}}\}$, where $s_{i,1} > s_{i,2} > \ldots > s_{i,{k_i}}$. For $j=1,\ldots, d$ denote by $p_j$ the following linear polynomial:
$$
p_j(\vec{x}) = (-s_{1,j} \tp x_1) \ta \ldots \ta (-s_{i,j} \tp x_i) \ta \ldots \ta (-s_{n,j} \tp x_n) \ta 0.
$$
Observe that $\vec{a} \in S_1 \times \ldots \times S_n$ is a root of $p_j$ if for some $i$ $a_i = s_{i,j}$ and for the rest of $i$ we have $a_i \leq s_{i,j}$.

Consider a degree $d$ polynomial
$
p(\vec{x}) = \bigodot_{j=1}^{d} p_{j}(\vec{x}).
$
Then from Lemma~\ref{lem:tropical_polynomials_product} we have that $\vec{a} \in S_1 \times \ldots \times S_n$ is a non-root of $p$ iff for all $i$ $a_i < s_{i,{d}}$. Thus the number of non-roots of $p$ is 
$
\prod_{i=1}^{n} \left( |S_i| - d\right).
$
This proves the upper bound.

For the lower bound, suppose there is a polynomial $p$ with the individual degrees $d$ that has more than $
\prod_{i=1}^{n} k_i - \prod_{i=1}^n \left(k_i-d\right)
$ roots in $S_1\times \ldots \times S_n$. Then the number of its non-roots in this set is at most $\prod_{i=1}^n \left(k_i-d\right) - 1$. Denote the set of all non-roots by $R$.

Consider a family of all the polynomials of the individual degree at most $k_i-d-1$ in variable $x_i$ for all $i$. Then their (common) support is of size $\prod_{i=1}^n \left(k_i-d\right)$. Since the size of the support is greater than $R$, by Theorem~\ref{thm:comb_general1} there is a polynomial $q$ with this support that vanishes on $R$.

Then, by Lemma~\ref{lem:tropical_polynomials_product} the non-zero polynomial $p\tp q$ vanishes on $S_1\times \ldots \times S_n$ and on the other hand has support 
$
\{0,\ldots, k_1-1\}\times \ldots \times \{0,\ldots, k_n-1\}.
$
This contradicts Theorem~\ref{thm:combinatorial_null}. Thus there is no such polynomial $p$ and the theorem follows.
\end{proof}

\section{Tropical Universal Testing Set} \label{sec:univ_set}

In this section we study the minimal size of a universal testing set for sparse tropical polynomials.
It turns out that in the tropical case there is a big difference between testing sets over $\bb{R}$ and $\bb{Q}$. Thus, we consider these two cases separately below.

Throughout this section we denote by $n$ the number of variables in the polynomials, by $k$ the number of monomials in them and by $s$ the number of points in a universal testing set.

\subsection{Testing sets over $\bb{R}$}

In this section we will show that the minimal size $s$ of the universal testing set over $\bb{R}$ is equal to $k$.

\begin{theorem} \label{thm:univ_set_r}
For tropical polynomials over $\bb{R}$ the minimal size $s$ of the universal testing set for polynomials with at most $k$ monomials is equal to $k$.
\end{theorem}

\begin{proof}
First of all, it follows from Theorem~\ref{thm:comb_general1} that for any set of $s$ points there is a polynomial with an arbitrary support having $k=s+1$ monomials that has roots in all $s$ points. Thus, the universal testing set has to contain at least as many points as there are monomials, and we have the inequality $s \geq k$.

Next we show that $s \leq k$. 
Consider a set of $s$ points $S = \{\vec{a}_1,\ldots, \vec{a}_s\} \in \bb{R}^n$ that have linearly independent over $\bb{Q}$ coordinates. Suppose we have a polynomial $p$ with $k$ monomials that has roots in all the points $\vec{a}_1,\ldots, \vec{a}_s$. We will show that $k \geq s+1$. Thus, we will establish that $S$ is a universal set for $k=s$ monomials.

Suppose the monomials of $p$ are $m_1,\ldots, m_k$, where $m_i(\vec{x})= c_i \tp \vec{x}^{J_i}$. Introduce the notation $p(\vec{a}_j) = \max_{i} (m_{i}(\vec{a}_j)) = p_j$. Since $a_j$ is a root, the value $p_j$ is achieved on at least two monomials.

Note that the monomial $m_i$ has the value $p_j$ in the point $\vec{a}_j$ iff 
$$
\langle\vec{a}_j, J_i\rangle + c_i = p_j.
$$

Now, consider a bipartite undirected graph $G$. The vertices in the left part correspond to monomials of $p$ ($k$ vertices). The vertices in the right part correspond to the points in $S$ ($s$ vertices). We connect vertex $m_i$ in the left part to the vertex $\vec{a}_j$ in the right part iff $m_i(\vec{a_j}) = p_j$.

Observe, that the degree of vertices in the right part is at least 2 (this means exactly that they are roots of $p$). 

Now, we will show that there are no cycles in $G$. Indeed, suppose there is a cycle. For the sake of convenience of notation assume the sequence of the vertices of the cycle is
$$
m_1, \vec{a}_1, m_2, \vec{a}_2, \ldots, m_l, \vec{a}_l.
$$
Note that since the graph is bipartite, the cycle is of even length.
In particular, for all $i=1,\ldots, l$ we have $m_i(\vec{a_i}) = p_i$, that is
\begin{equation} \label{eq:univ_on_r1_app}
\langle\vec{a}_i, J_i\rangle + c_i = p_i.
\end{equation}
Also for all $i=1,\ldots, l$ we have $m_{i+1}(\vec{a_i}) = p_i$ (for convenience of notation assume here $m_{l+1}=m_1$), that is
\begin{equation} \label{eq:univ_on_r2_app}
\langle\vec{a}_i, J_{i+1}\rangle + c_{i+1} = p_i.
\end{equation} 
Let us sum up all equations in~\eqref{eq:univ_on_r1_app} for all $i=1,\ldots,l$ and subtract from the result all the equations in~\eqref{eq:univ_on_r2_app}.
It is easy to see that all $c_i$'s and $p_i$'s will cancel out and thus we will have
$$
\langle \vec{a}_1, J_1\rangle - \langle \vec{a}_1, J_2\rangle + \langle \vec{a}_2, J_2\rangle - \langle \vec{a}_2, J_3\rangle + \ldots + \langle \vec{a}_l, J_l\rangle - \langle \vec{a}_l, J_1\rangle = 0.
$$
Since $J_1 \neq J_2$, we have a nontrivial linear combination with integer coefficients of the coordinates of vectors $\vec{a}_1,\ldots, \vec{a}_l$. Since the coordinates of these vectors are linearly independent over $\bb{Q}$, this is a contradiction.
Thus, we have shown that there are no cycles in $G$.

Therefore, the graph $G$ is a forest. Consider each of the trees of the forest separately. We will show that in each of these trees $T$ the number $L$ of vertices in the left part is greater than the number $R$ of vertices in the right part. Indeed, since the degree of each vertex in the right side is at least $2$, the number of edges in $T$ is at least $2R$. The number of vertices in a tree is by one greater than the number of edges. Thus, there are at least $2R+1$ vertices in $T$.
That is
$$
R+L \geq 2R+1,
$$
and thus $L \geq R+1$.
Since this holds for each tree, summing up these inequalities over all the trees we have
$$
k \geq s + 1.
$$

Thus, the set $S$ is a universal set against polynomials with $k=s$ monomials and the theorem follows.
\end{proof}

\subsection{Testing sets over $\bb{Q}$}

The main difference of the problem over the semiring $\bb{Q}$ compared to the semiring $\bb{R}$ is that now the points of the universal set have to be rational. 

In this section we consider, somewhat more generally, tropical polynomials with rational (possibly negative) powers of variables. We note that this does not actually affect the questions under consideration: for each such polynomial there is another polynomial with natural exponents with the same set of roots and the same number of monomials. Indeed, suppose $p$ is a polynomial with rational exponents. Recall that 
\begin{equation} \label{eq:polynomial_with_rational_coef}
p(\vec{x}) = \max(m_1(\vec{x}),\ldots, m_{k}(\vec{x})),
\end{equation}
where $m_1, \ldots, m_k$ are monomials. Recall that each monomial is a linear function over $\vec{x}$. Note that if we multiply the whole expression~\eqref{eq:polynomial_with_rational_coef} by some positive constant and add the same linear form $m(\vec{x})$ to all monomials, the resulting polynomial will have the same set of roots. Therefore, we can get rid of rational degrees in $p$ by multiplying $p$ by large enough integer, and then we can get rid of negative degrees by adding to $p$ a linear form $m$ with large enough coefficients.

Thus, throughout this section we consider polynomials with rational exponents.

It will be convenient to state the results of this section using the following notation. Let $k(s,n)$ be the minimal number such that for any set $S$ of $s$ points in $\bb{Q}^n$ there is a tropical polynomial on $n$ variables with at most $k(s,n)$ monomials having roots in all the points of $S$.  Note that there is a universal testing set of size $s$ for polynomials with $k$ monomials iff $k<k(s,n)$. Thus, we can easily obtain bounds on the size of the minimal universal testing set from the bounds on $k(s,n)$.

We start with the following upper bound on $k(s,n)$.
\begin{theorem} \label{thm:upper bound on k}
We have $k(s,n) \leq \left\lceil \frac{2s}{(n+1)}\right\rceil + 1$.

Equivalently, for the size of the minimal universal testing set the following inequality holds: $s\geq \frac{(k-1)(n+1)+1}{2}$.
\end{theorem}

We note that this theorem already shows the difference between universal testing sets over $\bb{R}$ and $\bb{Q}$ semirings.

\begin{proof}
Observe that two statements of the theorem are equivalent. Indeed, by our definition of $k(s,n)$ the first statement is equivalent to the inequality $k< \left\lceil \frac{2s}{(n+1)}\right\rceil + 1$, where $s$ is the size of the minimal testing set for polynomials with $k$ monomials. It is easy to see that this is true iff $s>(k-1)(n+1)/2$. The minimal integer $s$ for which this inequality holds is $s=\frac{(k-1)(n+1)+1}{2}$. Thus, the inequality is equivalent to $s\geq \frac{(k-1)(n+1)+1}{2}$. Thus, it remains to prove the first statement of the theorem.

We will show that for any set $S =  \{\vec{a}_1,\ldots, \vec{a}_s\}\subseteq \bb{Q}^n$ of size $s$ there is a nontrivial polynomial with at most $k= \lceil \frac{2s}{(n+1)}\rceil + 1$ monomials that has roots in all of the points in $S$. From this the inequalities in the theorem follow.

Throughout this proof we will use the following standard facts about (classical) affine functions on $\bb{Q}^n$.

\begin{claim} \label{cl:geometry}
Suppose $\pi$ is an $(n-1)$ dimensional hyperplane in $\bb{Q}^n$. Let $P_1$ be a finite set of points in one of the (open) halfspaces w.r.t. $\pi$ and $P_2$ be a finite set of points in the other (open) halfspace. Let $C_1$ and $C_2$ be some constants.
Then the following is true.
\begin{enumerate}
\item If $\vec{a}_1,\ldots, \vec{a}_n \in \pi$ are points in a general position in $\pi$ (that is, not lying in $(n-2)$-dimension linear space) and $p_1,\ldots, p_n$ are some constants in $\bb{Q}$, then there is an affine function $f$ on $\bb{Q}^n$ such that $f(\vec{a}_i)=p_i$ for all $i$, $f(\vec{x}) > C_1$ for all $\vec{x} \in P_1$ and $f(\vec{x})<C_2$ for all $\vec{x} \in P_2$.
\item If $g$ is an affine function on $\bb{Q}^n$ then there is another affine function $f$ on $\bb{Q}^n$ such that $f(\vec{x})=g(\vec{x})$ for all $\vec{x} \in \pi$, $f(\vec{x}) > C_1$ for all $\vec{x} \in P_1$ and $f(\vec{x})<C_2$ for all $\vec{x} \in P_2$.
\end{enumerate}
\end{claim}

The proof of the theorem is by induction on $s$. The base is $s=0$. In this case one monomial is enough (and is needed since we require polynomial to be nontrivial).

Consider the convex hull of points of $S$. Take a maximal dimension face $P$ of this convex hull. If $S$ is of dimension $n$, then $P$ is $(n-1)$-dimensional and if $S$ is of dimension less than $n$ we consider $P$ to be just the convex hull of $S$. For simplicity of notation assume that the points from $S$ belonging to $P$ are $\vec{a}_1,\ldots,\vec{a}_l$. 
Consider a ($(n-1)$-dimensional) hyperplane $\pi$ passing through $\vec{a}_1,\ldots, \vec{a}_l$. Since $P$ is a face of the convex hull of $S$ all the points  in $S' = \{\vec{a}_{l+1},\ldots,\vec{a}_s\}$ lie in one (open) halfspace w.r.t. $\pi$ (if $S$ is of dimension less than $n$, then $l=s$). 

Applying the induction hypothesis we obtain a polynomial $p'(\vec{x}) = \max_i m_i'(\vec{x})$ that has roots in all the points of $S'$. For $j=1,\ldots, l$ introduce the notation $p_j = p'(\vec{a}_j) = \max_{i} m_{i}(\vec{a}_j)$.

We consider three cases: $P$ contains all the points of $S$; $P$ contains not all the points of $S$ and $l \leq n$; $P$ contains not all the points of $S$ and $l > n$.

If $P$ contains all the points of $S$, then the polynomial $p'$ is obtained from the base of induction and consists of one monomial $m_1'$. Recall, that a monomial is just an affine function on $\bb{Q}^n$. Consider a new monomial $m(\vec{x})$ such that $m(\vec{x})=m_1'(\vec{x})$ on the hyperplane $\pi$, but $m(\vec{b}) \neq m_{1}'(\vec{b})$ for some $\vec{b} \notin \pi$. Then the polynomial $p = p' \ta m$ has roots in all the points of the hyperplane $\pi$ and thus in all the points of $S$. This polynomial has $2 \leq \left\lceil \frac{2s}{(n+1)}\right\rceil + 1$ monomials.

If $P$ contains not all the points of $S$, then the dimension of $P$ is $n-1$ (indeed, otherwise $P$ is not a face).

If additionally $l \leq n$, it follows that $l=n$.
Thus $\vec{a}_1,\ldots, \vec{a}_n$ are points in the general position in $\pi$. Thus due to the claim above we can pick a new monomial $m$ such that $m(\vec{a}_j) = p_j$ for all $j=1,\ldots,l$ and $m(\vec{a}_j) < p'(\vec{a}_j)$ for all $j > l$. Then the polynomial $p = p' \ta m$ has roots in all the points of $S$. This polynomial has 
$
1 + \left\lceil \frac{2(s-n)}{(n+1)}\right\rceil + 1 \leq \left\lceil \frac{2s}{(n+1)}\right\rceil + 1
$
monomials.

Now, if $l\geq n+1$ let $p_0 = \max_{j\leq l} p_j$.
Applying the claim above take a pair of new distinct monomials $m_1$ and $m_2$ such that $m_1(\vec{x})=m_2(\vec{x})=p_0$ for all $\vec{x} \in \pi$ and  $m_1(\vec{a}_j), m_2(\vec{a}_j) < p'(\vec{a}_j)$ for all $j > l$.
Then the polynomial $p = p' \ta m_1 \ta m_2$ has roots in all the points of $S$. This polynomial has at most
$
2 + \left\lceil \frac{2(s-n-1)}{(n+1)}\right\rceil + 1 = \left\lceil \frac{2s}{(n+1)}\right\rceil + 1
$
monomials.

In all three cases we constructed a polynomial with the desired number of monomials.
\end{proof}

The construction above leaves the room for improvement. 
For example, for the case of $n=2$ we can show the following.

\begin{theorem} \label{thm:univ_dim_2_upper}
For $n=2$ we have $k(s,2) \leq \left\lceil \frac{s}{2}\right\rceil + 1$.  For the size of a minimal universal set for polynomials in $2$ variables the following inequality holds: $s \geq 2(k-1) +1$.
\end{theorem}

\begin{proof}
The proof of equivalence of two statements in the theorem is analogous to the proof of the similar equivalence in Theorem~\ref{thm:upper bound on k}.

For the proof of the first statement again, we use the same strategy as in the proof of Theorem~\ref{thm:upper bound on k}. We perform the same case analysis on the induction step.
Note that in the first two cases the step of induction works.

Thus, the only remaining case is $l > 2$ and $P$ contains not all the points of $S$. There is a line $\pi$ in $\bb{Q}^2$ containing points $\vec{a}_1,\ldots, \vec{a}_l$ and such that all the points in $S \setminus \{\vec{a}_1,\ldots,\vec{a}_l\}$ are in one halfspace w.r.t. $\pi$. Consider the point of $S \setminus \{\vec{a}_1,\ldots,\vec{a}_l\}$ that is the closest one to the line $\pi$. Draw the line $\pi'$ parallel to $\pi$ through this point. If there are several points of $S$ on $\pi'$ consider the one that does not lie between two others. To simplify the notation let this vertex be $\vec{a}_{l+1}$. Denote the set of remaining vertices by $S'=\{\vec{a}_{l+2},\ldots, \vec{a}_s\}$ and 
apply the induction hypothesis to $S'$. 
As before let $p_j = p'(\vec{a}_j)$.

Consider a new monomial $m_1$ (recall that the monomial is just an affine function on $\bb{Q}^2$) such that $m_1(\vec{a}_{l+1}) = p_{l+1}$, $m_1(\vec{a}_{j}) \leq p_{j}$ for all $\vec{a}_j \in S \cap \pi'$, $m_1(\vec{a}_{j}) \leq p_{j}$ for all $\vec{a}_j \in S'\setminus\pi$ and  $m_1(\vec{a}_{j}) \geq p_{j}$ for all $j\leq l$. Note that this is possible by Claim~\ref{cl:geometry} since $\vec{a}_1, \ldots, \vec{a}_l$ and $S'\setminus\pi$ are situated in the opposite halfplanes w.r.t. $\pi'$. Finally, pick yet another new monomial $m_2$ such that $m_1(\vec{x})=m_{2}(\vec{x})$ for all $\vec{x} \in \pi$ and $m_{2}(\vec{a}_j) \leq p_j$ for all $j>l$. 
Then the polynomial $p = p' \ta m_1 \ta m_2$ has roots in all the points of $S$. This polynomial has at most 
$$
2 + \left\lceil \frac{s-4}{2}\right\rceil + 1 = \left\lceil \frac{s}{2}\right\rceil + 1
$$
monomials.
\end{proof}

Later we will show that this bound is tight.

We now proceed to lower bounds on $k(s,n)$. We start with the following non-constructive lower bound.
\begin{theorem} \label{thm:univ_set_nonconstructive}
We have $k(s,n) \geq \left\lceil \frac{s}{n+1}\right\rceil$.

Equivalently, for the minimal size of the universal testing set over $\bb{Q}$ we have $s \leq k(n+1) + 1$.
\end{theorem}

\begin{proof}
Observe that two statements of the theorem are equivalent. Indeed, by our definition of $k(s,n)$ the first statement is equivalent to the fact that for any $k$ and $s$ if $k < \left\lceil \frac{s}{n+1}\right\rceil$ then there is a testing set of size $s$ for polynomials with at most $k$ monomials. The inequality in this statement can be rewritten as $s \geq k(n+1) + 1$.  The statement then is equivalent to the fact that the minimal size of a testing set $s$ for polynomials with at most $k$ monomials satisfy the inequality $s \leq k(n+1) + 1$.

Next we prove the first statement of the theorem.
Within this proof we will temporarily switch to polynomials over $\bb{R}$. We also for the sake of this proof generalize powers of monomials to be real.
Suppose for any set $S = \{\vec{a}_1,\ldots, \vec{a}_s\} \in \bb{R}^n$ there is always a polynomial with $k$ monomials that has roots in all $s$ points.

The set of all tuples $\vec{a}_1,\ldots, \vec{a}_s$ of $s$ points in $\bb{R}^n$ forms an $sn$ dimensional space over $\bb{R}$.
Suppose a polynomial $p$ with monomials $m_1,\ldots, m_k$ has roots in all the points $\vec{a}_1,\ldots, \vec{a}_s$. This means that on each point $\vec{a}_j$ there are two monomials that has two equal values. 
By a \emph{configuration} we call an assignment to each point $\vec{a}_j$ of a pair of monomials $m_{i_1},m_{i_2}$ and a coordinate $l$ such that $m_{i_1}(\vec{a}_j) = m_{i_2}(\vec{a}_j)$ and the power of $x_l$ in $m_{i_1}$ is greater than the power of $x_l$ in $m_{i_2}$ by at least $1$ (we need this to ensure that $m_{i_1}$ and $m_{i_2}$ are distinct monomials). Any configuration is given by a set of tuples $(j,i_1,i_2, p)$, where $1\leq j\leq s$, $1 \leq i_1,i_2\leq k$ and $1 \leq l \leq n$, so there are finitely many configurations. 

Consider the $(sn+k(n+1))$-dimension space formed by tuples $\vec{a}_1,\ldots, \vec{a}_s$ and $J_{1},c_{1},\ldots, J_k, c_k$, where $J_i$ is the vector of powers of $m_i$ and $c_i$ is its constant term. For each configuration we can consider a semi-algebraic set (a set given by a finite Boolean combination of algebraic equations and inequalities) given by equations $m_{i_1}(\vec{a}_j)= m_{i_2}(\vec{a}_j)$ and inequalities $J_{i_1,l} - J_{i_2,l} \geq 1$ for all tuples $(j,i_1, i_2,l)$ in the configuration. By our assumption each point $(\vec{a}_1,\ldots, \vec{a}_s)$ lies in the projection of one of these semi-algebraic sets.

Note that in each point any of these semialgebraic sets have dimension at most $k(n+1) + s(n-1)$. Indeed, we can consider the following set of local coordinates. We include in this set all coordinates of $J_{1},c_{1},\ldots, J_k, c_k$ (there are $k(n+1)$ of them). For each $a_{j}$ we can consider the corresponding tuple $(j,i_1,i_2,l)$ and include in the set of local coordinates all coordinates of $a_j$ except the $l$-th coordinate. The $l$-th coordinate can be expressed from the others via the equation $m_{i_1}(\vec{a}_j)= m_{i_2}(\vec{a}_j)$ thanks to the inequality $J_{i_1,l} - J_{i_2,l} \geq 1$.

Thus each of our semi-algebraic sets is of dimension at most $k(n+1)-s(n-1)$. By Tarski's theorem a projection of a semialgebraic set is also a semi-algebraic set (see, e.g.~\cite{BasuPR06}) and the dimension does not increase after the projection. Thus by our assumption we can cover all points $(\vec{a}_1,\ldots, \vec{a}_s)$ of $sn$-dimensional space by a finite number of dimension at most $k(n+1)-s(n-1)$. If there is an inequality
$$
k(n+1) + s(n-1) < sn
$$
between the dimensions, this is impossible,  and so there is a tuple $S$ (over $\bb{R}$) such that for any polynomial $p$ with at most $k$ monomials there is a non-root for $p$ in $S$. Our next goal is to prove that there
exists a tuple $S$ over $\bb{Q}$ satisfying the latter property.

For this, consider our semi-algebraic sets in coordinates $\vec{a}_1,\ldots, \vec{a}_s$ and consider their closures. These are still semi-algebraic sets and they are still of dimension a most $k(n+1) + s(n-1)$. So the complement of their union in $\bb{R}^{sn}$ (that is nonempty due to the inequality between dimensions) is an open set and contains each point $(\vec{a}_1,\ldots, \vec{a}_s)$ with a neighborhood. It remains to observe that this neighborhood contains a point with rational coordinates.
\end{proof}

The lower bound on $k(s,n)$ in Theorem~\ref{thm:univ_set_nonconstructive} is not constructive.
In the next section we present some constructive lower bounds. 
For this we establish a connection of our problem to certain questions in discrete geometry.

\subsection{Constructive Lower Bounds} 

Suppose for some set of points $S = \{\vec{a}_1, \ldots, \vec{a}_s\} \subseteq \bb{Q}^n$ there is a polynomial $p$ with monomials $m_1,\ldots, m_k$ that has roots in all the points of $S$.

Recall that the graph of $p$ in $(n+1)$-dimensional space is a piece-wise linear convex function. Each linear piece being a polyhedron corresponds to a monomial and roots of the polynomial are the points of non-smoothness of this function, so the roots of $p$ are the boundaries of these polyhedra.
Consider the set of all the roots of $p$ in $\bb{Q}^n$. They partition the space $\bb{Q}^n$ into at most $k$ convex (possibly unbounded) polyhedra. Each polyhedron corresponds to one of the monomials $m$ and consists of all the points
$\vec{a}\in \bb{Q}^n$ such that $m(\vec{a})=p(\vec{a})$. Note that any two of these polyhedra are \emph{separated} by a hyperplane: if the polyhedra correspond to monomials $m_i$ and $m_j$, then the first one lies in the halfspace $m_i(\vec{x}) \leq m_{j}(\vec{x})$ and the second one lies in the halfspace $m_{i}(\vec{x})\geq m_j(\vec{x})$.

Consider the polyhedron corresponding to the monomial $m_i$. Consider all the points in $S$ that lie on its boundary and consider their convex hull. We obtain a smaller (bounded) convex polyhedron that we will denote by $P_i$. 

Thus starting from $p$ we arrive at the set of pairwise separated polyhedra $P_1,\ldots, P_k$ with vertices in $S$ and not containing any points of $S$ in the interior (here we consider polyhedra in $n$-dimensional space and their $n$-dimensional interiors, that is $\vec{a}$ is in the interior if its $n$-dimensional $\eps$-neighborhood is contained in the polyhedron for small enough $\eps>0$; it might be that some polyhedra have empty interior). The statement that $p$ has roots in all the points of $S$ means that each point in $S$ belongs to at least two of the polyhedra $P_1,\ldots, P_k$. 

Motivated by this analysis we introduce the following definition. Given a set of $s$ points in $n$-dimensional space by a \emph{double covering} of points of $S$ by bounded convex polyhedra we call a collection of polyhedra $P_1,\ldots, P_k$ such that they are pairwise separated and each point in $S$ lies on the ($(n-1)$-dimensional) boundary of at least two polyhedra. Here we say that the polyhedra $P$ and $Q$ are \emph{separated} if there is a hyperplane $\pi$, such that $P$ and $Q$ lie in different closed halfspaces w.r.t. $\pi$. In particular, $P$ and $Q$ can intersect only by the points of $\pi$ and thus only by their ($(n-1)$-dimensional) boundary. 
The \emph{size} of the covering is the number $k$ of the polyhedra in it.

From the discussion above we have that if we will construct a set $S$ of points that does not have a double covering of size $k$ it will follow that $S$ is a universal set for $k$ monomials. 

The similar notion of single covering has been studied in the literature~\cite[page~367]{BMP05book}. Given a set of $s$ points in $n$-dimensional space by a \emph{single covering} of points of $S$ by bounded convex polyhedra we call a collection of polyhedra $P_1,\ldots, P_k$ they are pairwise separated and each point in $S$ lies on the ($(n-1)$-dimensional) boundary of one of the polyhedra. The \emph{size} of the single covering is the number $k$ of the polyhedra in it.  

Denote by $k_1(s,n)$ the minimal number of polyhedra that is enough  to single cover any $s$ points in $n$ dimensional space.
Denote by $k_2(s,n)$ the minimal number of polyhedra that is enough  to double cover any $s$ points in $n$ dimensional space.

The above analysis results in the following theorem.
\begin{theorem} \label{thm:covering_connection}
$k(s,n) \geq k_2(s,n) \geq k_1(s,n)$.
\end{theorem}

For single coverings the following results are known.
Let $f(n)$ be the maximal number such that any large enough $n$-dimensional set of points $S$ contains a set of $f(n)$ points that lie on the boundary of some convex polyhedron and on the other hand there are no other points in $S$ in the interior of this polyhedron. The function $f(n)$ was studied but is not well understood yet. It is known~\cite{Valtr92} that the function is at most factorial in $n$. We can however observe the following.

\begin{lemma} \label{lem:single_covering_trivial}
For large enough $s$ we have that $k_1(s,n) \geq s/f(n)$.
\end{lemma}

\begin{remark}
We observe that our definitions of $f(n)$ and $k_1(s,n)$ slightly differ from the ones of~\cite{Valtr92} and~\cite{BMP05book}. On one hand, in~\cite{Valtr92} and~\cite{BMP05book} it is required that the points in $S$ are in the general position. On the other hand, it is required that the points lie not only on the boundary of the poyhedra, but in its vertices and polyhedra in the single covering are not allowed to intersect. However, our definitions are equivalent to the definitions of~\cite{Valtr92} and~\cite{BMP05book}. Indeed, on one hand, our notions are not more general for the case when the points in $S$ are in the general position, since we can always restrict polyhedra to their convex hulls (and in case some point is covered more than once in the covering by polyhedra, just remove it from all of the polyhedra but one). On the other hand, the same values of $f(n)$ and $k_1(s,n)$ as for the points in general position can be achieved for arbitrary set of points. Indeed, having the set $S$ of points not in the general position, we can move them slightly to make them to be in the general position, find the desired polyhedra, restrict them to the convex hulls of points they are covering and move the points back (along with their convex hulls). It is easy to see that if the movement of points was small enough the polyhedra will satisfy all the desired properties (points remain on the boundary of polyhedra and the polyhedra remain separated).
\end{remark}

\begin{proof}[Proof of Lemma~\ref{lem:single_covering_trivial}]
Consider a large enough set of $s$ points in general position with no empty polyhedra of size $f(n)+1$. Then in any covering each polyhedron can contain at most $f(n)$ points, hence the lower bound follows.
\end{proof}

It is known~\cite{Valtr92} that $f(3)\geq 22$. Thus we get that $k_1(s,3) \geq s/22$ for large enough $s$.

It is also known~\cite{Urabe99} that $\lceil s/2(\log_2 s +1) \rceil \leq k_1(s,3) \leq \lceil 2s/9 \rceil$. For $n=2$ there are linear upper and lower bounds known~\cite{Urabe96}. For an arbitrary $n$ in~\cite{Urabe99} an upper bound $k_1(s,n) \leq 2s/(2n+3)$ is shown and $k_1(s,n) = \lceil s/2n\rceil$ is conjectured.

As a trivial corollary of Lemma~\ref{lem:single_covering_trivial} we obtain the following.

\begin{corollary} \label{cor:covering1}
For large enough $s$ we have that $k(s,n) \geq s/f(n)$.
\end{corollary}

\begin{remark}
We note that although Corollary~\ref{cor:covering1} gives a lower bound on $k(s,n)$ for large enough $s$, it can be restated for all $s$. Suppose $s_0$ is the smallest $s$ for which the inequality in the lemma holds. Note that there is a trivial bound $k(s,n)\geq 1$. Consider $g(n) = \max{(f(n),s_0)}$. Then we have $k(s,n) \geq s/g(n)$.
\end{remark}

\begin{lemma} \label{lem:covering2}
$k_1((n+2)s,n) \geq k_2(s,n)$. 
\end{lemma}

\begin{proof}
Consider a set of $s$ points and substitute each point by the set of vertices of a small enough $n$-dimensional simplex and by its center. Thus we substitute each point by $n+2$ points and obtain $(n+2)s$ points as a result. Consider a single covering of these points of size $k_1((n+2)s,n)$. None of the polyhedra in this cover can contain the whole simplex and its center. Thus, each simplex contains vertices of at least two polyhedra. Merging all the points of each simplex back into one point results in a double covering of the original set of the same size (assuming the simplices are small enough).
\end{proof}

Overall, we have a sequence of inequalities 
$
k(s,n) \geq k_2(s,n) \geq k_1(s,n) \geq k_2(\frac{s}{n+2},n).
$
We do not know how large $k(s,n)$ can be compared to $k_1(s,n)$ and $k_2(s,n)$. 

However this connection helps us to show that the lower bound on the size of universal testing set we have established before for the case of $n=2$ is tight.

\begin{theorem} \label{thm:univ_set_dim_2}
We have $k(s,2) \geq k_2(s,2) \geq \left\lceil \frac s2 \right\rceil + 1$.

Therefore, for $n=2$ the size of the minimal universal testing set is equal to $s = 2k-1$.
\end{theorem}

The remaining part of this section is devoted to the proof of Theorem~\ref{thm:univ_set_dim_2}.

The second part of the theorem follows from the first part and Theorem~\ref{thm:univ_dim_2_upper} immediately. 

Thus, it remains to show that $k_2(s,2) \geq \left\lceil \frac s2 \right\rceil + 1$.

As a universal set with $s$ points in $\bb{Q}^2$ we will pick the set of vertices of an arbitrary convex polygon $M$.

Suppose we have some double covering of the vertices of $M$ by $k$ polygons. Among these polygons let us distinguish the set $E$ of those that are edges of $M$ and the set $T$ of other polygons. Denote $|E|=k_1$ and $|T|=k_2$, thus $k=k_1+k_2$. 
Denote by $W$ the sum of the number of vertices in all polygons in $T$.

We will show the following lemma.

\begin{lemma} \label{lem:weight_bound}
For $s\geq 2$ we have $W \leq s+2k_2 -2$.
\end{lemma}

First let us show why this lemma is enough to finish the proof of the lower bound on $k_2(s,2)$.

Note that each polygon from $E$ has two vertices. Thus, the sum of the number of vertices in all polygons in $E$ is $2k_1$. The sum of the number of vertices in all polygons in $T$ by definition is $W$. Each vertex of $M$ should be a vertex for at least two polygons in $E$ and $T$. Thus, the sum of the numbers of vertices in all the polygons in $E$ and $T$ is at least $2s$. 
Thus, we get that
$$
2s \leq 2k_1 + W \leq 2 k_1 + s+2k_2 -2,
$$
where the second inequality follows from Lemma~\ref{lem:weight_bound}.
From this we get 
$$
k = k_1 + k_2 \geq \frac s2 +1.
$$
Since $k$ is an integer we have $k \geq \left\lceil\frac s2 \right\rceil +1$ and the theorem follows.

Thus it remains to prove the lemma.
\begin{proof}[Proof of Lemma~\ref{lem:weight_bound}]

The proof is by induction on $s$.

The base case is $s=2$ (a degenerate polygon). Then $T = \emptyset$, $k_2=0$, $W=0$ and the inequality follows.

Consider $s\geq 3$. If $k_2=0$, then $W=0$ and the inequality obviously holds. Suppose $k_2\geq 1$ and pick an arbitrary polygon $P$ in $T$. Suppose there are $r$ vertices in $P$. Then $P$ splits the remaining part of $M$ into $r$ separate convex polygons (possibly degenerate, that is with just 2 vertices) $M_1\ldots, M_r$. Denote the number of vertices in them by $s_1,\ldots, s_r$ respectively. Note that 
\begin{equation}\label{eq: number_of_vertices}
s_1+\ldots+s_r=s+r.
\end{equation}
Suppose in polygons $M_1,\ldots, M_r$ there are $t_1,\ldots, t_r$ polygons in $T$ respectively. Denote the sets of these polygons by $T_1,\ldots, T_r$ respectively. Then
\begin{equation}\label{eq: number_of_regions}
t_1+\ldots+t_r=k_2-1.
\end{equation}
Suppose the sum of the numbers of vertices in $T_i$ is $W_i$ for $1 \leq i \leq r$. Then
\begin{equation}\label{eq: weight}
W_1+\ldots+W_r=W-r.
\end{equation}
By the induction hypothesis for any polygon $M_i$ we have the following inequality:
\begin{equation} \label{eq:induction_hypothesis}
W_i \leq s_i+2t_i -2.
\end{equation}
Adding up inequality~\eqref{eq:induction_hypothesis} for all $i=1,\ldots, r$ and using \eqref{eq: number_of_vertices}-\eqref{eq: weight} we get
$$
W-r \leq (s+r) +2(k_2-1) -2r,
$$
i.~e.
$$
W \leq s+2k_2-2
$$
and the lemma follows.
\end{proof}

\section{Tropical $\tau$-conjecture} \label{sec:tau}

Since in the max-plus semiring the distributivity holds ($a\tp (b \ta c) = a \tp b \ta a \tp c$) and since the definition of the root does not depend on the specific representation of a polynomial, we can consider representation of polynomials by arbitrary tropical formulae. Even more, we can consider its representation by \emph{tropical circuit}.

A tropical circuit $C$ in variables $x_1, \ldots, x_n$ is a directed acyclic graph each vertex of which is of in-degree $0$ or $2$. Each vertex of in-degree 0 is labeled by either a variable, or a constant in the semiring. Each vertex of in-degree $2$ is labeled by one of the operations $\ta$ or $\tp$. Labeled vertices of a circuit are called gates. Each gate computes a tropical polynomial defined inductively in the natural way. One of the gates is distinguished as the output gate. The circuit computes a polynomial that is computed by its output gate. The size of the circuit $|C|$ is the number of gates in it. 

A formula is a special case of a circuit in which every (not output) gate has out-degree 1. A standard observation is that this definition of a formula is equivalent to a common definition of a formula as an expression consisting of variables, constants, operations and brackets. 

The classical $\tau$-conjecture addresses the question of how many integer roots can a classical polynomial of one variable have in terms of the size of the minimal classical algebraic circuit computing this polynomial~\cite{Blum98}. In the tropical case however roots of any polynomial can be made integer by a simple modification of the polynomial. 

\begin{lemma} \label{lem:roots_integer}
For any tropical polynomial $p$ of one variable computable by a circuit (or a formula) of size $s$ there is a tropical polynomial $p'$ of one variable computable by a circuit (a formula) of size $s$ that has the same number of roots as $p$ and all roots of $p'$ are integer. 
\end{lemma}

\begin{proof}
Consider a tropical circuit $C$ of size $s$ computing the tropical polynomial $p(x)$ of one variable $x$. We first show that there is a tropical circuit of size $s$ computing a troipcal polynomial with the same number of roots such that all constants used in the circuit are rational.

Consider all constants $a_1,\ldots, a_k$ used in $C$ and substitute them by fresh formal variables $c_1,\ldots, c_k$. We are going to construct a system of linear inequalities with rational coefficients on $c_1,\ldots, c_k$ that reflects that the root structure of the polynomial (in variable $x$) computed by the circuit is the same as for $p$. We then observe that this system has rational solution.

We can view the output of the circuit as a tropical polynomial over $x$ which coefficients are tropical polynomials over $c_1,\ldots, c_k$ (basically, we are considering the decomposition over the variable $x$ of the polynomial over the variables $x, c_1, \ldots, c_k$). That is, each monomial of this polynomial over $x$ is $b_i \cdot x + q_i(c_1,\ldots, c_k)$ for $i=1,\ldots, m$, some integers $b_i$'s as tropical exponents of $x$ and some tropical polynomials $q_i$'s as coefficients. For each $q_i(c_1,\ldots,c_k)$ consider its monomial $l_i(c_1,\ldots, c_k)$ on which the minimum of $q_i$ is attained in the point $(a_1,\ldots, a_k)$. Add to our system of inequalities all inequalities stating that $l_i(c_1,\ldots, c_k)$ is less or equal that each of the other monomials of $q_i$. Since each monomial is a linear form with integer coefficients, each inequality is a linear inequality with integer coefficients.

Next, consider linear forms 
$$
g_i(x, c_1, \ldots, c_k)=b_i x + l_i(c_1,\ldots, c_k)
$$ 
for $i=1,\ldots, m$ with integer coefficients. For $(c_1,\ldots, c_k) = (a_1,\ldots, a_k)$ these expressions in variable $x$ form linear pieces of the graph of the function computed by the circuit. For each pair of forms $g_i$ and $g_j$ we have that either intersection point of $g_i(x, a_1, \ldots, a_k)$ and $g_j(x, a_1, \ldots, a_k)$ (as linear functions in one variable $x$) lies below some linear function $g_{i'}(x,a_1,\ldots, a_k)$, and then this point is not a root of the output of the circuit $C(x)$, or the intersection point lies above (or lies on) all other linear functions $g_{i'}(x,a_1,\ldots, a_k)$ and then it is a root of $C(x)$. We add all these relations between all triples of linear forms $g_i(x,c_1,\ldots, c_k)$, $g_j(x,c_1,\ldots, c_k)$ and $g_{i'}(x,c_1,\ldots, c_k)$. Each of these relations can be clearly expressed as a linear inequality in variables $c_1,\ldots, c_k$ with rational coefficients. Indeed, the intersection point of $g_i$ and $g_j$ has $x$ coordinate
$$
x= \frac{l_i(c_1,\ldots, c_k) - l_j(c_1,\ldots, c_k)}{b_j-b_i}.
$$
Substituting it into $g_i$ and $g_{i'}$ and fixing an inequality $\leq$ or $\geq$ between them depending on which one should be above the other we obtain the desired linear inequality in $c_1,\ldots, c_k$.

Overall, we obtain the system of linear inequalities with rational coefficients with variables $(c_1,\ldots, c_k)$ such that if some vector  $(c_1,\ldots, c_k)\in \bb{K}^k$ satisfies them, the function computed by $C(x)$ with constants $(c_1,\ldots, c_k)$ has the same number of roots as $p$. This linear system has a solution: $(c_1,\ldots, c_k)=(a_1,\ldots, a_k)$. Thus, it has a rational solution as well. Substitute this rational solution as constants in $C$.

Observe that once all coefficients in the polynomial of one variable are rational, the roots are rational as well (as intersection points of two linear functions with rational coefficients).

Finally, consider a tropical circuit $C$ computing the polynomial $p$ and construct a new circuit $C'$ of size $s$ that differs from $C$ in that every constant used in $C$ is multiplied by the same factor $\alpha$. Denote by $p'$ the polynomial computed by $C'$. Then we claim that for any $x$ 
\begin{equation} \label{eq:roots_integer}
p'(\alpha \cdot x)=\alpha \cdot p(x).
\end{equation}
In particular $a$ is a root of $p$ iff $\alpha \cdot a$ is a root of $p'$.

The proof of~\eqref{eq:roots_integer} is by simple induction on the size of the circuit: the equation is trivial for variables and constant and all operations allowed in the circuit preserve the equation.

To finish the proof of the lemma, consider a circuit $C$ with rational coefficients and multiply all constants in it by a suitable factor to make all roots integer.
\end{proof}

By Lemma~\ref{lem:roots_integer} studying the number of integer roots of tropical formulae and circuits is equivalent to studying the number of arbitrary roots in them.

Let $\# f$ denote the number of roots of a tropical univariate polynomial~$f$.

\begin{lemma}
For any tropical univariate polynomials $f$ and $g$ we have
\begin{itemize} \label{lem:formula_props}
\item $\# f\ta g \le \# f + \# g +1$;
\item $\# f\tp g \le \# f + \# g$;
\item $\# f^{\tp k} = \# f$.
\end{itemize}
\end{lemma}

\begin{proof}
Recall that a tropical polynomial of one variable is a piece-wise linear convex function on the set $\bb{R}$ and the roots of the polynomial are the non-smoothness points of this function, that is the number of linear pieces minus 1.

Note that $f\ta g$ is just $\max{(f,g)}$ in classical terms, so we have that $f \ta g$ can have as its linear pieces only the parts of linear pieces of $f$ and $g$. Thus, the number of linear pieces of $f\ta g$ is at most the sum of the number of linear pieces of $f$ and $g$ and the inequality for the number of roots follows.

Note that $f \tp g$ is just $f+g$ in classical terms. So, each point of non-smoothness of $f \tp g$ must be a point of non-smoothness of at least one of the functions $f$ and $g$. So the inequality for the number of roots follows.

Finally, observe that $f^{\tp k}$ is just $k\cdot f$ in the classical terms and this function has exactly the same set of non-smoothness points.
\end{proof}

\begin{lemma} \label{lem:formula}
If a polynomial $f$ is given by a formula $C$ then $\# f \leq |C|$.
\end{lemma}

\begin{proof}
The proof of this lemma is a trivial induction on the size of the formula. The step of induction easily follows from Lemma~\ref{lem:formula_props}.
\end{proof}

Thus we have shown that tropical polynomials computable by polynomial size formulae have at most polynomially many roots (and thus at most polynomially many integer roots).

\begin{remark}
Note that Lemma~\ref{lem:formula} extends to the setting in which there are exponentiation gates in the formula that do not add to the size of the circuit.
\end{remark}

Now we proceed to the case of max-plus polynomial circuits. It turns out that the answer to the question here is opposite (with respect to formulae) and we will construct an example of a circuit with exponentially many integer roots. To do this it is convenient to extend the notion of tropical polynomials and consider tropical rational functions.

For this we introduce operation of tropical division: for $x, y \in \bb{K}$ let
$$
x \td y= x-y.
$$

\begin{definition}
A function $f\colon \bb{K}^n \to \bb{K}$ is a tropical rational function if it can be expressed as a well-formed formula with variables $x_1,\ldots, x_n$, constants in $\bb{K}$ and operations $\ta, \tp$ and $\td$.
\end{definition}

The next two lemmas are not new~\cite{CuninghameG80,Ovchinnikov02}, but we present the proofs for the sake of completeness.
\begin{lemma} \label{lem:rational_simple}
Any non-trivial tropical rational function is a piece-wise linear function.
\end{lemma}

\begin{proof}
The statement of the lemma is true for variables and constant and piece-wise linearity is clearly preserved under the operations $\ta, \tp$ and $\td$.
\end{proof}

A point $\vec{a}\in \bb{K}^n$ is a \emph{root} of the tropical rational function $f$ if it is the point of non-smoothness of $f$, that is if $p$ belongs to at least two linear pieces of $f$.

\begin{remark}
We note that for the case of tropical rational functions of one variables there is usually a distinction between points of non-smoothness in which the change of slope
is positive and points in which the change of slope is negative. The former are usually called roots, and the latter are called poles (see e.g.~\cite{HalburdS09}). This distinction is not important for us, so we prefer to use the word `roots' for both cases.
\end{remark}

It is not hard to see that tropical rational functions can be expressed as a tropical division of two tropical polynomials.

\begin{lemma} \label{lem:rational_normal_form}
For any tropical rational function $f$ (with arbitrary number of variables) there are tropical polynomials $p$ and $q$ such that
$$
f = p \td q.
$$

For tropical rational function $f$ of one variable for any root $a$ of $f$ consider the intervals $(b,a)$ and $(a,c)$ on which $f$ is linear. If the slope of $f$ on $(a,c)$ is greater than the slope on $(b,a)$, then $a$ is a root of $p$. If on the other hand the slope of $f$ on $(a,c)$ is smaller than the slope on $(b,a)$, then $a$ is a root of $q$.
\end{lemma}

\begin{proof}
The proof of the first statement of the lemma is by the simple induction.

If $f$ is a variable or a constant then just let $p=f$ and $q=0$.

If $f$ is obtained by one of the operations from tropical rational functions $f_1$ and $f_2$ we can prove the statement of the lemma just translating usual operations with fractions to tropical setting. More specifically, consider tropical polynomials $p_1, p_2, q_1, q_2$ such that $f_1 = p_1 \td q_1 = p_1 - q_1$ and $f_2 = p_2 \td q_2 = p_2 - q_2$.

If we have that $f = f_1 \tp f_2$, then
$$
f = f_1 + f_2 = p_1 - q_1 + (p_2 - q_2) = (p_1 + p_2) - (q_1 + q_2)
$$
and we can let $p = p_1 \tp p_2$ and $q = q_1 \tp q_2$.

If $f = f_1\td f_2$, then analogously we can let $p = p_1 \tp q_2$ and $q = q_1 \tp p_2$.

If $f = f_1 \ta f_2$, then we have
\begin{align*}
f &= \max{(f_1, f_2)} = \max{(p_1 - q_1, p_2 - q_2)} \\
&= \max{\left(p_1 + q_2 - (q_1 + q_2), p_2 + q_1 - (q_1 + q_2)\right)}\\
& = \max{(p_1 + q_2, p_2 + q_1)} - (q_1 + q_2)
\end{align*}
and we can let $p = p_1 \tp q_2 \ta p_2 \tp q_1$ and $q = q_1 \tp q_2$.

For the second part of the proof we argue by a contradiction. Assume that the slope of $f$ on $(a,c)$ is greater than the slope of $f$ on $(b,a)$, but $a$ is not a root of $p$. Then, for small enough $\eps$ we have that on $(a -\eps, a+\eps)$ the function $p$ is linear. Note however, that on this interval $q$ is convex and $f$ is concave. This contradicts equation $f = p - q$. The case when the slope of $f$ on $(a,c)$ is smaller than the slope of $f$ on $(b,a)$ is completely analogous.
\end{proof}

\begin{remark}
We note that the representation of the tropical rational function $f$ (with arbitrary number of variables) in the form as in Lemma~\ref{lem:rational_normal_form} is not unique.
\end{remark}

\begin{remark}
We note that it is not hard to show that if in Lemma~\ref{lem:rational_simple} we additionally assume that $f$ is continuous and all slopes of $f$ are integer, then the converse is also true. That is, any continuous piecewise linear function with integer slopes can be expressed as a difference of tropical polynomials (see e.g.~\cite{CuninghameG80,Ovchinnikov02}).
\end{remark}


Analogously to tropical circuits we can introduce \emph{rational tropical circuits}. The only difference is that now the operation $\td$ is also allowed and thus the circuit computes a tropical rational function. 

There is a close connection between tropical rational circuits and tropical circuits.

\begin{lemma} \label{lem:rational_to_poly}
Suppose a tropical rational circuit $C$ computes a tropical rational function $f$. Then there are tropical polynomials $p$ and $q$ such that $f = p \td q$ and $p $ and $q$ can be computed by tropical circuits (without $\td$ operation) of size at most $4 |C|$.
\end{lemma}

\begin{proof}
Each gate $g$ of $C$ computes some tropical rational function $f_g$. A simple inductive argument shows that we can introduce $p_g$ and $q_g$ such that $f_g = p_g \td q_g$ and reconstruct a circuit in such a way that for each gate the circuit computes $p_g$ and $q_g$ and the circuit does not use $\td$ operation. 

Indeed, this is trivial for input gates. For the step of induction consider a gate $g$ and assume that the statement is established for all previous gates. The gate $g$ has two inputs $h_1$ and $h_2$. By induction hypothesis in the reconstructed circuit we have gates $p_{h_1}$, $q_{h_1}$, $p_{h_2}$ and $q_{h_2}$ such that $h_1= p_{h_1} \td q_{h_1}$ and $h_2=p_{h_2} \td q_{h_2}$. To construct $p_g$ and $q_g$ we can just use simulations of operations with rational functions from the proof of Lemma~\ref{lem:rational_normal_form}. For example, if $g = h_1 \ta h_2$ we can immediately set $q_{g} = q_{h_1} \tp q_{h_2}$. To compute $p_g$ we introduce intermediate gates $g_1 = p_{h_1} \tp q_{h_2}$ and $g_2 = p_{h_2} \tp q_{h_2}$. Then $p_{g} = g_1 \ta g_2$. The cases $g = h_1 \tp h_2$ and $g = h_1 \td h_2$ are even simpler. Note that to simulate each gate of the original circuit at most four operations $\ta$ and $\tp$ are required.
\end{proof}

Now we are ready to provide an example of tropical rational functions in one variable that can be computed by small circuits and on the other hand have many roots. This example is an adaptation of the construction from~\cite{montufar2014}.

Consider 
\begin{equation} \label{eq:rational_example_0}
f_0(x)= \max(-2x+1, 2x - 1).
\end{equation}
For $i=1,2, \ldots$ define the function iteratively:
\begin{equation}\label{eq:rational_example}
f_{i} =f_{0}\circ f_{i-1} = \max(-2f_{i-1} + 1, 2f_{i-1} - 1).
\end{equation}
Note that $f_0, f_1, \ldots$ are tropical rational functions.

\begin{lemma}
The function $f_n$ can be computed by a rational tropical circuit of size $O(n)$.
\end{lemma}

\begin{proof}
The proof of this lemma is by simple induction: just note that due to~\eqref{eq:rational_example} to compute each next $f_n$ from the previous one we need constantly many operations.
\end{proof}

On the other hand, the function $f_n(x)$ has many roots.

\begin{theorem} \label{thm:rational_roots}
The function $f_{n}(x)$ is equal to $1$ in all points of the set  $S_{1,n} = \{\frac{k}{2^{n}} \mid k = 0, 1,\ldots, 2^n\}$ and is equal to $0$ in all points of the set $S_{0,n} = \{ \frac{k}{2^n} + \frac{1}{2^{n+1}} \mid k=0,\ldots, 2^{n}-1\}$. The function is linear between each two consecutive points of $S_{1,n} \cup S_{0,n}$. Thus, $f_{n}(x)$ has $2^{n+1}-1$ roots on the interval $(0,1)$, namely the roots are $(0,1) \cap (S_{1,n} \cup S_{0,n})$.
\end{theorem}

\begin{proof}
The proof is by induction on $n$. For $n=0$ the theorem is easy to check directly.

Suppose the statement of the theorem is true for $f_n$ and consider $f_{n+1}$.
Observe that the function $g =2f_{i-1}-1$ is equal to $1$ on $S_{1,n}$, is equal to $-1$ on $S_{0,n}$ and is linear in between of the points $S_{1,n} \cup S_{0,n}$. The function $h = -2f_{i-1}+1$ is symmetrical to $g$: it is equal to $-1$ on $S_{1,n}$, is equal to $1$ on $S_{0,n}$ and is linear in between of the points $S_{1,n} \cup S_{0,n}$.

We have that $f_{n+1} = \max{(g,h)}$ and thus it is equal to $1$ in the points of $S_{1,n} \cup S_{0,n} = S_{1, n+1}$. On each interval between the consecutive points of $S_{1,n} \cup S_{0,n}$ one of two functions $g$ and $h$ goes from the value $-1$ to $1$ and the other goes from $1$ to $-1$. Thus they intersect in the middle of the interval, where both functions are equal to $0$. Thus, we have that $f_{n+1}$ is equal to $0$ in all points $\{ \frac{k}{2^{n+1}} + \frac{1}{2^{n+2}} \mid k=0,\ldots, 2^{n+1}-1\} = S_{0,n+1}$ and is linear on each interval between consecutive points of $S_{1,n+1} \cup S_{0,n+1}$.
\end{proof}

\begin{remark}
Another example of a tropical rational function with a number of roots exponential in the circuit size can be found in~\cite{AllamigeonBGJ18}.
\end{remark}

Now we are ready to prove the main result of this section.

\begin{theorem} \label{thm:tropical_circuits}
There is a sequence of tropical polynomials $r_1(x),\ldots, r_n(x), \ldots$ of one variable such that they are computable by a tropical circuit of size $O(n)$ and on the other hand $r_{n}(x)$ has at least $2^n$ roots.
\end{theorem}

\begin{proof}

Consider the function $f_{n}$. This function is computable by a tropical rational circuit of size $O(n)$. By Lemma~\ref{lem:rational_to_poly} there are two polynomials $p_n$ and $q_n$ such that $f_n = p_n \td q_n$ and $p_n$ and $q_n$ are computable by a tropical circuit of size $O(n)$.

By Theorem~\ref{thm:rational_roots} $f_{n}$ has roots at each point in $\{1/2^{n+1}, 2/2^{n+1}, 3/2^{n+1}, \ldots, (2^{n+1}-1)/2^{n+1}\}$. By Lemma~\ref{lem:rational_normal_form} $q_n$ has roots at each point in $\{1/2^{n+1}, 3/2^{n+1}, 5/2^{n+1}, \ldots, (2^{n+1}-1)/2^{n+1}\}$, while $p_n$ has roots at each point in $\{2/2^{n+1}, 4/2^{n+1}, 6/2^{n+1}, \ldots, (2^{n+1}-2)/2^{n+1}\}$. So, for $r_n$ one can pick $q_n$.
\end{proof}

Recall that by Lemma~\ref{lem:roots_integer} it follows that there is also a sequence of polynomials with the same circuit-size and with the same number of roots, that are all integer (it is enough to substitute constant $1$ in \eqref{eq:rational_example_0}, \eqref{eq:rational_example} by $2^{n+1}$).


\subsection*{Acknowledgements} 


We would like to thank anonymous reviewers for numerous helpful comments.

{\small
\bibliographystyle{abbrv}
\bibliography{bib/tropical}
}

\end{document}